\documentclass[reqno, 10pt]{article}
\usepackage[utf8]{inputenc}
\usepackage{mathtools,mathrsfs,amssymb,amsthm,amsmath,bm}
\usepackage{hyperref}
\usepackage{url}
\usepackage[a4paper, left = 1.5in, right = 1.5in, top =1.5in, bottom = 2in]{geometry} 
\usepackage{dsfont}   
\usepackage{enumitem}
\usepackage[british]{babel}
\usepackage{comment}

\newtheorem{thm}{Theorem}
\newtheorem{cor}{Corollary}
\newtheorem{pro}{Proposition}
\newtheorem{fact}{Fact}

\theoremstyle{definition}
\newtheorem{defi}{Definition}
\newtheorem{exam}{Example}
\newtheorem*{rem}{Remark}
\newtheorem*{note}{Note}

\def\[#1\]{\begin{align*}#1\end{align*}}

\newcommand{\R}{\mathbb{R}}
\newcommand{\N}{\mathbb{N}}
\newcommand{\ceq}{\coloneqq}

\newcommand{\eps}{\varepsilon}
\newcommand{\E}{\mathbb{E}}
\renewcommand{\P}{\mathbb{P}}
\def\B{\mathscr{B}}
\newcommand{\lrtx}[1]{\ \text{#1} \ }

\newcommand{\K}{\mathbb{K}}
\renewcommand{\t}{\top\!}
\newcommand{\df}{\mathop{}\!\mathrm{d}}
\newcommand{\D}{\mathbb{D}}
\newcommand{\M}
{\mathbb{M}}
\newcommand{\F}
{\mathscr{F}}
\newcommand{\lbr}{\underline}
\newcommand{\dmid}{\,\|\,}
\newcommand{\MD}{\mathcal{M}}
\newcommand{\I}{\mathds{1}}
\newcommand{\argmin}{\mathrm{argmin}}

\begin{document}
\title{An Intrinsic Treatment  of Stochastic Linear Regression}
\author{Yu-Lin Chou\thanks{Yu-Lin Chou, Institute of Statistics, National Tsing Hua University, Hsinchu 30013, Taiwan,  R.O.C.; Email: \protect\url{y.l.chou@gapp.nthu.edu.tw}.
The author wishes to express gratitute for the comments received for the first  version.}}
\date{}
\maketitle

\begin{abstract}
\fontsize{8}{7.5}\selectfont
Linear regression is perhaps one of the most popular  statistical concepts,
which permeates almost every scientific field of study.   Due to the technical  simplicity and   wide applicability of linear regression, attention is almost always quickly  directed to the algorithmic or computational side of linear regression. In particular,
the underlying mathematics of stochastic linear regression itself as an entity usually gets either a   peripheral  treatment or a relatively in-depth but \textit{ad hoc} treatment depending on the type of concerned problems;
in other words, compared to the  extensiveness  of the study of  mathematical properties of the ``derivatives''  of stochastic linear regression such as the least squares estimator,
the mathematics of stochastic linear regression itself seems to have not yet received a due  intrinsic treatment.
Apart from the conceptual importance,
a consequence of an insufficient or  possibly inaccurate understanding of stochastic linear regression would be the recurrence   for the role of  stochastic linear regression in the important (and more sophisticated)  context of  structural equation modeling to be misperceived or taught in a misleading way.  
We believe this pity is rectifiable when the fundamental concepts are correctly classified.
Accompanied by some illustrative, distinguishing examples and counterexamples, we intend to pave out the mathematical framework for stochastic linear regression, in a rigorous but non-technical way, by giving new results and pasting together  several fundamental known  results that are, we believe, both enlightening and conceptually useful, and that had not yet been systematically documented in the related literature. 
As a minor contribution, the way we  arrange the fundamental known results would be the first attempt in the related literature.\\

{\noindent {\bf Keywords:}}
concept classification;  conditional expectation; 
counterexamples in statistics; orthogonal projection; stochastic linear regression\\
{\noindent {\bf MSC 2020:}} 62J05;  	62A99   
\end{abstract}

\section{Introduction}
We are attempting to correctly classify the concept of stochastic linear regression and
several ubiquitous related  concepts,
and are much less concerned with problems of practical interest  regarding stochastic linear regression.
Figuratively,  
we wish to ``embed'' the concept of stochastic linear regression in mathematics,
in particular in probability theory; 
we wish to ``inject''  the concept of stochastic linear regression into mathematics in a ``structure-preserving'' way. 

In view of the unusual topics concerning the present paper, 
we ask for the reader's extendability to emcompass the following ostensible digression,
which, we believe, facilitates  communicating our purposes  and serves as an organic part of the present paper. 
From a postmodern viewpoint,
the topics concerning the present work  seem to lie between ``meta-statistics'' and statistical theory, 
a no man's land.

Our results  claim no   ``usefulness'' 
in the customary sense, 
and, as Flexner \cite{f} argued,
a ``useless'' knowledge need not turn out to be useless as long as we are willing to move from a local viewpoint to a global, sustainable one. While being useful is certainly not a ``sin'',
we believe that not a few forms of usefulness would be dangerous in the long run. The potentially dangerous senses of usefulness may be best described by the quote ``$\cdots$ are not fit for heaven, but on earth they are most useful. $\cdots$. 'T is the same with mules, horses, $\cdots$'' from the great poet Mary A. Evans (George Eliot) \cite{e}.

The indicated purpose is not as exotic as it sounds once we see that it is simply a natural part of  developments of a mathematical theory, and statistics, as a branch of mathematics, has since Fisher's  modern  initiation  already implicitly moved towards obtaining a  unified embedding in the sense that every statistical object is defined as some  mathematical object. 
We invite the reader to think for your reference about the definition of a random sample,
of an estimator, of a test, or of a random field,
although the last one would be more of a probabilistic flavor.
A natural, logical conclusion drawn from the phenomenon is that modern statistics tends to embed their concerned concepts in mathematics.
We remark in passing that this embedding   property is a prerogative of statistics, which is not shared by engineering fields or even physics. For an engineering field, the reason is evident; mathematics plays in the field a role that  facilitates  modeling works. As to physics,
although to a great extent  it may be embedded in mathematics (e.g. the concept of spacetime),
there are many concepts in physics that may not be reasonably taken as a mathematical object (e.g. the concept of mass or of mutual interaction).

On the other hand,
for a knowledge system to qualify as a science,
a necessary condition is for the system to be developed towards  internal unification. We might mention and re-appreciate Euclid's genius --- the invention of axiomatics --- that is acknowledged as the first complete attempt of human mind to logically rearrange the then scattered ``mathematical facts'' in such a way that a sane person may reason out for herself the known or unknown mathematical facts  under a given set of few pre-defined rules. Mathematics is then  well-eligible for being a science in various senses; Pythagoras theorem, or any mathematical theorem in general, has since been not just an interesting recurrent phenomenon, nor just a useful trick for engineering purposes, nor just a wise opinion from the esteemed scholars, nor some truth that is unfathomable to the civilians and only ``owned'' by the rich and powerful. 
The Euclid's invention thus demystifies a significant aspect of the nature of mathematics, and independentizes mathematical activities,
making them essentially not exclusively belonging to any social class. 
Further,
we believe, even a person who does not work in any particular scientific field would expect a knowledge system, 
if generally accepted as a science,
to be much more than just a cookbook or a collection of ``very useful methods''
whose deeper  connections are left unorganized.
This  non-philosophy --- satisfaction of a collection of very useful methods without  caring to seek after the deeper connections --- seems to be a prerogative of the business, industrial world;
after all, by nature  they seek profit (no moral judgement is implied), and hence  a sense of immediate satisfaction. 
However, if we acknowledge that a science is supposed to seek truth, then it would be unjustified to stay at the satisfaction level of the business-industririal activities.   

This tendency --- for a statistical object to be defined as a mathematical object --- is also an  enlightened, edified  movement as, in terms of  mathematics, 
we can save ourselves from spending energy on the philosophical or semantical queries into what we are really talking about by focusing on the functional properties of the concepts of our concern. For instance, rather than arguing what a random variable really is and then defining it,
we define a random variable by requiring what a random variable should do, or, equivalently, by requiring what we can do with a random variable. 
Evidently,
this ``epistemological'' approach is recurrent in mathematics and  a signature therein, and may be fairly referred to as a mathematical approach. 
The result is  more than beneficial; 
as well-known and well-received, 
it turns out that the well-established mathematical objects ---  (probability) measure and measurable function 
--- may be  used to define a random variable,
and so the statistical object --- random variable --- 
is in this sense embedded in mathematics. 
It then follows that the fundamental  statistical object ---  random sample --- also becomes a mathematical object.
We also wish to point out another more than familiar event of mathematization  by embedding a long-time vague object in  mathematics: Kolmogorov's measure-theoretic treatment of probability. The concept of probability had  long been arguably a controversial object; but Kolmogorov's  mathematical astuteness led him to recognize that the mathematical object ---  measure --- just serves the purpose of delineating what   probability should do, and the  nice ramifications of this Kolmogorov's embedding are stunning. 

Another kind of benefits obtainable from   establishing a suitable embedding is about conceptual coherency and clarity. 
From a panoramic view,   it is evidently desirable for the theory of any  mathematical science to admit as few  ambiguities as  possible, so that,
for example,
the understanding of any concept thereof does not depend on the interpretation of any individual therein, 
which in turn ensures the efficiency and quality of the   scientific communications.    

Although most statistical concepts are defined as some  mathematical concept,  
the important  statistical concept --- stochastic linear regression --- is an exception.
When it comes to stochastic linear regression,
the customary treatment seems  to be \textit{ad hoc}  depending on the problem at hand;
for instance, sometimes stochastic linear regression is associated with  conditional expectation,
sometimes it is associated with orthogonal projection, sometimes it is nearly taken to be an arbitrary  ``linear model'',
sometimes it is associated with algorithms such as least squares (and hence treated as a technique), and sometimes it is left tacitly understood as a string of symbols representing ``the familiar form requiring no further elaboration''.

And, usually, in teaching materials 
the particular aspects of stochastic linear regression are stressed without a caveat nor a further elaboration for a full, more complete picture; and none of the partial descriptions establishes stochastic linear regression precisely as a mathematical object in a reasonable way.
Besides,
these partial descriptions of stochastic linear regression, each of which captures a component of the concept of stochastic linear regression, 
are, however, independent in the (weak) sense that no two of them are equivalent.
It is not difficult to write down a justification for this observation. 
Among the partial descriptions,
a less evident non-equivalence would be affine  conditional expectation and linear orthogonal projection, the latter being equivalent to the uncorrelatedness between error term and regressor(s) under ``very'' mild, reasonable assumptions.   
We will prove this particular non-equivalence later on.
It seems that an example of this non-equivalence, apparently heuristically enlightening,  rarely appears in the related literature.

Thus the term ``stochastic linear regression'' seems to be just a placeholder such that, depending on the problem at hand,
it could mean different things;
the most significant possible meanings of ``stochastic linear regression''  are as listed above. 
It is clear that,
to embed the concept of stochastic linear regression in mathematics,
we cannot rely on the last three of the partial descriptions --- stochastic linear regression as a ``model'', as a technique, and as a string of symbols or ``equations'';
the first two themselves are not mathematical objects,
and the last one is ``morally'' a mathematical object but out-of-context. That an equation interpretation is not suitable for describing stochastic linear regression may be seen as follows. Indeed, an equation in mathematics is taken as a predicate, which is a well-established, clear concept in logic, 
and to solve an equation means to find some element of a given set such that the predicate is true of the element.
For example, a heat equation ``$\partial_{t}u = \partial^{2}_{x}u$''
(considered on a suitable subset of $\R^{2}$)
is, according to the convention, 
precisely the predicate ``the $(0,1)$-partial derivative of a function  being equal to the $(2,0)$-partial derivative of the function''; and one may ask if there is some element of a given class of functions on the given domain such that the predicate is true of it, i.e. such that it satisfies the equation.
But the purposes associated with stochastic linear regression never involve solving for ``$\beta$'' directly from the given equations; instead, it is solving for ``the optimal `$\beta$'" from a moment condition derived from the given equations that is of concern. 
Moreover,
embedding stochastic linear regression in mathematics in terms of moment equations is not advisable as we are then led back to meet the non-equivalence between affine  conditional expectation and linear orthogonal projection.

Although affine conditional expectation and linear orthogonal projection are not equivalent, and none of them alone may fully equate to the concept of stochastic linear regression,
for a comparison 
we might add that they both are mathematical objects.
If $(\Omega, \F, P)$ is a probability space, if $Y$ is an $L^{1}$ random variable on $\Omega$,
and if $X$ is a random variable on $\Omega$,
then,
since the $P$-indefinite integral $(Y\df P)|_{\sigma(X)}$ of $Y$ restricted to the sigma-algebra $\sigma(X) \subset \F$ generated by $X$ is absolutely continuous with respect to $P|_{\sigma(X)}$,
the ($P|_{\sigma(X)}$-essential) Radon-Nikodym derivative of the measure $(Y\df P)|_{\sigma(X)}$ with respect to $P|_{\sigma(X)}$  exists,
and,
upon identifying two $P|_{\sigma(X)}$-almost sure equal $\sigma(X)$-measurable random variables  $\Omega \to \R$ with each other,
one may define the conditional expectation of $Y$ given $X$ as the thus obtained Radon-Nikodym derivative $D_{P|_{\sigma(X)}}(Y\df P)|_{\sigma(X)}$. 
If $\E(Y \dmid X)$ is not essentially constant,
and if the Doob-Dynkin function,
or the so-called  regression function,
of $Y$ given $X$,
being a function $\R \to \R$ such that the composition of it circ $X$ is the conditional expectation of $Y$ given $X$,
is affine,
then we obtain an example of affine  conditional expectation. 
On the other hand, if $X,Y \in L^{2}(P)$, and if $X$ is not essentially zero,  then the linear orthogonal projection of $Y$ given $X$ is precisely the random variable $\beta X$ with $\beta \in \R$ being the solution of the equation $\E X(Y-Xb) = 0$.

 We believe the unsatisfactory or improvable \textit{status quo} of the concept of stochastic linear regression is rectifiable. Inspecting the various special senses attached to stochastic linear regression, we have noticed that the ``definition'' of stochastic linear regression seems to depend on the context under consideration and hence on the  concerned problem. 
This suggests that there seems no intrinsic treatment for the important concept of stochastic linear regression, which is a pity. 
Our usage of  ``intrinsic'' coincides with the usual usage in mathematics (and even with philosophy such as the field of epistemology; e.g. Lewis \cite{l}) 
in a broad sense,
and seeking intrinsic properties certainly gains insight into the objects of interest, and hence is itselt intrinsically interesting.

For elementary examples in mathematics, in (linear) algebra the dimension of a vector space,
as any two bases of the space have the same cardinality, 
is an intrinsic property of the  vector space itself in the sense that the dimension of the  vector space does not depend on the choice of bases; in analysis,
 an $L^{p}$-metric may be made  well-defined (from being a pseudo-metric to being a metric) in terms of the $L^{p}$-norm and the    equivalence classes of $L^{p}$ functions (with respect to the equivalence relation of almost everywhere equality) by noticing that the resulting metric is  intrinsic in the sense that it  does not depend on the choice of the representatives; and in geometry, the dimension of a (topological) manifold,
 as a Euclidean space is homeomorphic to a Euclidean space precisely when their dimensions agree, 
 is an intrinsic property of the manifold in the sense that the dimension of the manifold does not depend on the choice of its  atlases.    
 
 Without implying any ``indoctrination'',
 we intend to suggest a reasonable, intrinsic look at the concept of stochastic linear regression,
 with the hope that the aforementioned issues may begin to be settled in a satisfactory, unified way, and without claiming a supreme generality encompassing all known types of ``regression''. At any rate, our framework is general enough to cover the interesting cases so as to be conceptually enlightening, and is at the same time sufficiently special to be tractable and informative without loss of practical meaningfulness. 
 
The generic idea and the nice consequents of the proposed intrinsic treatment of stochastic linear regression may be sketched as follows. 
By delineating the requirements of  suitable strength for stochastic linear regression with  the  requirements kept as few as possible such that both the problems of ``parameter learnability''
and of the quality of estimation are satisfactorily taken care of in theory,
we leverage the fact that every $L^{2}$ space is an inner product space to develop the concept of stochastic linear regression as a suitable class of probability measures defined on the Borel sigma-algebra of Euclidean subsets.
One may then obtain a somewhat unified viewpoint of stochastic linear regression. 
In between the developments,  we will also give new results and discuss informative, simple examples and counterexamples to clear up some mythologies pertaining to stochastic linear regression.

\section{The Treatment}
\subsection{Notation and Terminology}
If $m \leq l$ 
are natural numbers, 
if $1 \leq t_{1} < \cdots < t_{m} \leq l$ are natural  numbers,
and if $I \ceq \{ t_{1}, \dots, t_{m} \}$,
we will denote by $\pi_{I}$ the natural projection
$(z_{1}, \dots, z_{l}) \mapsto (z_{t_{1}}, \dots, z_{t_{m}})$ from $\R^{l}$ onto $\R^{m}$.
Thus $\pi_{I} = \pi_{\lbr{I}}$ when $I = \lbr{I}$; the definition of $\pi_{I}$ arranges the elements of $I$ from the smallest one to the greatest one.
If $I$ is a singleton, 
say $I = \{ 1 \}$,
we will write $\pi_{1}$ for $\pi_{\{ 1 \}} = \pi_{I}$. The domain of a natural projection $\pi_{I}$ will always be uniquely determined by context. Both the notation for natural projection and the  terminology for $\pi_{I}$  follow Billingsley \cite{b}. 

If $l \in \N$,
we denote by $\B_{\R^{l}}$ the Borel sigma-algebra generated by the usual topology of the Euclidean space  $\R^{l}$.

If $(\Omega, \F)$ is a measurable space,
we denote by $\Pi(\F)$ the collection of all probability measures defined on $\F$.
Thus $\Pi(\B_{\R^{l}})$ is for every $l \in \N$ the collection of all probability measures defined on $\B_{\R^{l}}$.
If $z \in \R^{l}$, the symbol $\D^{z}$ denotes the Dirac measure (degenerate distribution) $\B_{\R^{l}} \to \{ 0, 1\}, B \mapsto \I_{B}(z)$ ``concentrated on'' $\{ z \}$; 
so $\D^{z} \in \Pi(\B_{\R^{l}})$ for all $z \in \R^{l}$.

If $\Omega, \lbr{\Omega}$ are arbitrary sets,
and if $f: \Omega \to \lbr{\Omega}$,
we will write $f^{-1)}$ for the pre-image map $2^{\lbr{\Omega}} \to 2^{\Omega}, A \mapsto \{ x \in \Omega \mid f(x) \in A \}$ induced by $f$.

If $(\Omega, \F, P)$ is a probability space,
and if $Z: \Omega \to \R$ is a random variable,
we will frequently denote by $P_{Z}$ the induced probability measure of $P$ by $Z$, i.e. $P_{Z} \equiv P \circ Z^{-1)}$. Thus $P_{Z}$ is the (probability) distribution of $Z$. If $\P$ is the distribution of $Z$, we may sometimes  write $Z \sim \P$. If $\P$ is not the distribution of $Z$, we sometimes write $Z \not\sim \P$.

The symbol $\R_{+}$ denotes the set $\{ x \in \R \mid x \geq 0 \}$; and $\R_{++}$ denotes $\{ x \in \R \mid x > 0 \}$.

If $(\Omega, \F, \M)$ is a measure space,
and if $A \in \F$,
we denote by $\M\rceil_{A}$ the measure $\lbr{A} \mapsto \M(A \cap \lbr{A}), \F \to \R_{+} \cup \{ +\infty \}$. 
The measure $\M$ is said to be \textit{concentrated on} $A$ if and only if $\M = \M\rceil_{A}$ on $\F$. This justifies the statement that a Dirac measure $\D^{z}$ is concentrated on $\{ z \}$ for every suitable $z$. 
Our use of ``concentrated on'' has its roots in Rudin \cite{r}; and the corresponding notation is adapted from Federer \cite{fd}.
Since $\M\rceil_{A} = (\M\rceil_{A})\rceil_{A}$,
the measure $\M\rceil_{A}$ is concentrated on $A$. 
A random variable whose distribution is concentrated on a Borel subset $B$ of $\R$ will also be said to be concentrated on $B$. 
Thus a random variable not concentrated on any singleton subset of $\R$ is precisely a non-degenerate random variable. 
A random variable not concentrated on a singleton $\{ x \} \subset \R$ will also be said to be \textit{not essentially} $x$. We remark: If $X$ is a random variable, and if $\P$ is the distribution of $X$, then $\P$ is not concentrated on a singleton $\{ x \}$ of $\R$ if and only if $\P ( \{ x \} ) < 1$, which holds if and only if $\P \neq \D^{0}$, which holds if and only if $X \not\sim \D^{0}$.

The notation ``$\rceil_{A}$'' is also applied to collections of sets.
If $\F$ is a sigma-algebra of subsets of $\Omega$,
the symbol $\F\rceil_{A}$ denotes the relative  sigma-algebra $\{ A \cap \lbr{A} \mid \lbr{A} \in \F \}$ of $A$.

If $l \in \N$, 
and if a context is in the presence of a matrix operation such as transposition, then an  $l$-tuple is to be taken as an $l \times 1$ matrix. For instance, if $x,y \in \R^{l}$, then $x^{\t}y$ is the sum of the products of the $i$th components of $x$ and $y$, and $xy^{\t}$ is the matrix whose $(i,j)$-entry is $x_{i}y_{j}$.       

We will identify two random variables (on the same probability space),
if equal almost surely, with each other.

As there are different terminologies employed to address a measurable function having finite integral,
ranging over  $\{$``exist'',
``integrable'',
``summable'' $\}$,
we will often refer to a random variable having finite integral (finite mean) as an $L^{1}$ random variable. For a random variable having finite $L^{p}$-norm (finite $p$-th [raw] moment) with $1 \leq p \leq +\infty$,
the same rule applies. The underlying probability space may depend on but can be determined in terms of the context.

If $\Omega$ is a probability space,
if $Y$ is an $L^{1}$ random variable on $\Omega$, and if $X$ is a random vector on $\Omega$ with $l$ components (including the case where $l=1$),
we will write $\E(Y \dmid X) \ceq \E (Y \dmid \sigma(X))$; and we will denote by $\E(Y \mid X): \R^{l} \to \R$ 
the corresponding Doob-Dynkin (regression) function $f: \R^{l} \to \R$ such that $\E(Y \dmid X) = f \circ X$ on $\Omega$. Thus 
\[
\E(Y \dmid X) = \E(Y \mid X) \circ X
=
(\E(Y \mid X)(x))_{x \in \R^{l}}\circ X
\]
on $\Omega$. The domain of the function $\E(Y \dmid X)$ is $\Omega$, which is not necessarily $\R^{l}$; but the domain of $\E(Y \mid X)$ is $\R^{l}$.
Although in general we will refer to $\E(Y \dmid X)$ as the conditional expectation (of $Y$ given $X$) and to $\E(Y \mid X)$ as the regression function (of $Y$ given $X$),
sometimes we will also refer to a regression function as a conditional expectation. But this mixed usage will not cause any confusion. 
Moreover,
for what it is worth,
the random variable $Y$ is said to be \textit{mean independent} of $X$ if and only if $\E (Y \dmid X) = \E Y$.
Thus, if $Y$ is centered, i.e. if $\E Y = 0$,
then for $Y$ to be mean independent of $X$ means $\E(Y \dmid X) = 0$. The expectation of a random vector or a matrix of random variables is always understood componentwisely. For example, if $X$ is a random vector with each component $X_{j}$ being $L^{1}$, then $\E X$ denotes the real vector $(\E X_{1}, \dots, \E X_{k}) \in \R^{k}$.  

For our purposes, we will not take a specific definition of linear orthogonal projection; we deliberately let context and our arguments jointly  determine its role in stochastic linear regression. For the less experienced, our doing so helps build the idea of the underlying mathematical structure of stochastic linear regression; for the experienced, our doing so will hardly cause any confusion and, hopefully, will  clarify some less noticed aspects of stochastic linear regression. We will not always use the modifier ``linear'' when referring to linear orthogonal projection; but, as usual, context matters. The same considerations apply to the term ``coefficient of linear orthogonal projection''.

We will denote by $\K$ the covariance operator. If $Y$ is a random element of $\R^{q}$ with each component being $L^{2}$, and if $X$ is a random element of $\R^{l}$ with each component being $L^{2}$, then $\K(X,Y) \ceq \E (X - \E X)(Y - \E Y)^{\t}$.

Since we intend to connect the known results together whenever suitable,
we will use ``Fact'',
 instead of the usual ``Theorem'' or ``Proposition'', to state known results.

\subsection{Heuristics}
Regardless of the context where one  speaks of stochastic linear regression, the common fundamental material,
although usually off-stage,
is an unknown distribution $\P \in \Pi(\B_{\R^{1+k}})$. Here $k \in \N$ is given by the problem under consideration. We say ``given'' as we are considering the ``population'' situation prior to a confrontation with data,
so that revising the choice of $k$ is beyond the scope.
The probability measure $\P$ governs the behavior of the $k+1$  variates whose  statistical relationship interests the researcher. Specifically,
the researcher at least  wishes to investigate how a random variable $Y$ with distribution $\P_{\pi_{1}}$, 
which represents the dependent variate of interest to her,
depends on a linear combination of $k$ random variables  $X_{j}$,
representing the covariates of interest to her,
each of which has  distribution $\P_{\pi_{j+1}}$ with $1 \leq j \leq k$.
To make sense,
the random variables $Y, X_{1}, \dots, X_{k}$ certainly have to be defined on the same probability space at the outset; but to assume so  is always  realistic,
and to do so is always mathematically possible. 
For our purposes, taking $Y$ to be $\pi_{1}$ and each $X_{j}$ to be $\pi_{j}$ are mathematically just fine. 
And we remark that the possible presence of the constant covariate is readily taken care of by employing the Dirac measure $\D^{1}$ concentrated on $\{ 1 \}$. 

Now, to ensure that a meaningful result may be obtained out of observations on the $k+1$ variates, 
the probabilistic behavior of the random variables $Y, X_{1}, \dots, X_{k}$ cannot be arbitrary.
The researcher then needs to seek  conditions under which she can be assured that 
\begin{itemize}
\item[i)] 
a ``meaningful'' linear statistical relationship really exists and is actually  ``learnable'' from data,
i.e. some  linear statistical relationship between $Y$ and $X_{1}, \dots, X_{k}$ exists such that its interpretation makes good sense, and this relationship is uniquely determined by the distribution $\P$ of the random variables so that any suitable  transformation of data drawn from $\P$ will not  approximate the relationship vacuously;
\item[ii)] the probability that an estimation of the unique relationship makes the correct decision is well-controlled when the data are sufficiently nice and many.
\end{itemize}

The first requirement is intrinsic to the random variables $Y, X_{1}, \dots, X_{k}$,
independently of the probabilistic mechanism governing how data are generated (e.g. stationarity and  ergodicity).
The second requirement certainly depends more on the probabilistic mechanism that generates data,
and so it is more of a technical consideration to allow probability limit theorems such as laws of large numbers to work.
In short, 
the two requirements are the minimum requirements such that the first one prevents the researcher's study  from being an alchemy and the second one is necessitated by asking for a reasonable quality of estimation of the linear statistical relationship.

Although sufficient conditions ensuring the two requirements are well-known, most of the existing conditions are too strong for the two requirements (certainly, the existing sufficient conditions are also intended to take care of other desired properties.). For the second requirement, requiring $Y$ and each $X_{j}$ to be $L^{2}$ suffices, which allows of an application of a  weak law or even of Kolmogorov's strong law for well-behaved data such as independent identically distributed (i.i.d.) data. 

A set of  weak sufficient conditions for the first requirement is more interesting. 
For most of the time, the condition $\E (Y - X^{\t}\beta \dmid X) = 0$ (when making sense) is used along with a regularity condition (e.g. orthogonality) on $\{ X_{1}, \dots, X_{k} \}$ guaranteeing that there is exactly one $\beta \in \R^{k}$ such that $\E(Y - X^{\t}\beta \dmid X) = 0$,
so that $X^{\t}\beta$ is precisely the linear orthogonal projection of $Y$ given $X_{1}, \dots, X_{k}$.
The mean independence assumption of error term is a  possible factor of the etymology of stochastic linear regression, or linear regression in general. 

Nevertheless, although the mean independence assumption of error term is innocuous for multi-normal random vectors,
and might  as well be imposed based on a background structural theory whenever suitable, 
there is no reason why an arbitrary  random vector $(Y, X^{\top})$ with $Y$ being $L^{1}$  should serve that $\E(Y \mid  X)$ is affine. For instance, if $X$ is an $L^{2}$, non-degenerate random variable, and if $Y \ceq X^{2}$,
then $\E (Y \dmid X) = \E (X^{2} \dmid X) = X^{2}$; and the square function $x \mapsto x^{2}$ on $\R$ is not affine.
Further,
although $\E(\eps \dmid X) = 0$ implies $\E X\eps = 0$ provided that $\eps, X$ are random variables such that $\eps, X\eps$ are both $L^{1}$,
the converse is not true.
Before we show the falsehood of the converse, we remark that the falsehood is actually not surprising as the mixed moment $\E X\eps$ 
is in some sense a modulus of linear dependence between $X$ and $\eps$, while the  true dependence between $X$ and $\eps$ can certainly be wildly  nonlinear, so that their regression function $\E (\eps \mid X)$ can also take a wild form.

\begin{thm}[orthogonality without mean independence]
If $X$ is an $L^{3}$, non-degenerate random variable,
and if $\E X^{3} = 0$, then there is some $L^{1}$, non-degenerate  random variable $\eps$ on the same probability space such that $\E X \eps = 0$ and $\E (\eps \dmid X) = \eps$.
\end{thm}
\begin{proof}
Let $\eps \ceq X^{2}$. Then $\eps$ is $\sigma(X)$-measurable, and $\eps$ is $L^{1}$ by Jensen's inequality. Moreover,
we have $\E(\eps \dmid X) = \E(X^{2} \dmid X) = \eps$.
But by assumption we also have
$\E X\eps = \E X^{3} = 0$.
\end{proof}
We remark that, under the assumptions of Theorem 1, the random variables  $\eps, X\eps$ are both $L^{1}$; so Theorem 1 disproves  that  orthogonality implies mean independence in a \textit{bona fide}  way.

Thus counterexamples to the statement  that orthogonality implies mean independence are in fact abundant:

\begin{exam}
If $X \sim N(0,1)$,
and if $\eps \ceq X^{2}$,
then $\E X \eps = \E X^{3} = 0$;
but 
\[
\E (\eps \dmid X) = X^{2} \sim \chi^{2}(1),
\] 
which is certainly  not degenerate.

Slightly wilder  examples can be constructed easily.
For instance,
let $\Omega \ceq [0,1] \times \R$, and probabilitize $\Omega$ with respect to the evident product Borel sigma-algebra of $\B_{\R}\rceil_{[0,1]}$ 
(the Borel subsets of $[0,1]$)
and $\B_{\R}$ by the product probability measure of the Rademacher distribution and the standard Gaussian distribution,
so that $\pi_{1} \sim \frac{1}{2}(\D^{-1} + \D^{1})$
and $\pi_{2} \sim N(0,1)$, and $\pi_{1}$, $\pi_{2}$ are independent random variables.
The existence of a nontrivial  Rademacher random variable,
i.e. of a random variable has $\frac{1}{2}(\D^{-1} + \D^{1})$ as its distribution,
is well-known and may be constructed in a  non-artificial way by considering the dyadic expansions of elements of $[0,1]$.
If $X \ceq \pi_{2}$, and if  $\eps \ceq \pi_{1} + \pi_{2}^{2}$, then $\E X\eps = 0$ by the independence of $\pi_{1}$ and $\pi_{2}$. Moreover, since $\eps \in L^{1}(\Omega)$ by Minkowski's inequality,
from independence we also have $\E (\eps \dmid X) = \pi_{2}^{2} \sim \chi^{2}(1)$, which is never degenerate. \qed \end{exam}

From Theorem 1 it follows immediately that
\begin{cor}[non-equivalence between orthogonal projection and conditional expectation]
There are continuum-many random elements $(Y, X)$ of $\R^{2}$ such that the orthogonal projection and conditional expectation of $Y$ given $X$ exist and disagree.
\end{cor}
\begin{proof}
Indeed, the proof of Theorem 1 applies to any $L^{3}$  random variable with symmetric probability density function. If $X \sim N(0, \sigma^{2})$, and if $Y \ceq X + X^{2}$,
then $X$ is the orthogonal projection of $Y$ given $X$. Since $Y$ is then $L^{1}$ by Minkowski's inequality, the conditional expectation $\E (Y \dmid X)$ exists and is $= X + X^{2}$. 

Since $\R_{++}$  is in bijection with $\R$,
there are continuum-many choices of $\sigma$.
\end{proof}

Therefore, to take care of the first minimum requirement for stochastic linear regression, that there exists exactly one ``meaningful'' statistical relationship between $Y$ and $X_{1}, \dots, X_{k}$ in terms of linear combination of the covariates $X_{1}, \dots, X_{k}$,
the usual mean independence assumption $\E(Y - X^{\t}\beta \dmid X) = 0$ is much too strong with respect to mathematical considerations. 
It follows that,
even without any reference to data,
the  conditional expectation interpretation of stochastic linear regression is distorted. 
We might add that the  interpretation distortion means that, even in the event that the estimated orthogonal projection of $Y$ given $X$ passes all the tests and diagnostics, this estimated orthogonal projection may very well have little to do with the conditional expectation of $Y$ given $X$, and so it would  barely make sense to attach a sense of  effect on  the (conditional) average behavior of $Y$ to the estimated coefficient of orthogonal projection, although it makes every sense to view the estimated coefficient as an effect on the behavior of $Y$. 

However,  the mathematical remarks above do not necessarily always negate the legitimacy of the mean independence assumption of error term, and hence of the conditional expectation interpretation, 
of stochastic linear regression. We notice that we did not make any \textit{a priori}  assumption restricting how $\eps$ and $X$ are related \textit{a priori}, which certainly opens up a variety of possibilities. So a moral conveyed by Corollary 1 is this, that, 
in  practice,  what one necessarily learns via stochastic linear regression is not the conditional expectation of the involved random variables, unless there is   further information indicating the ``true'' dependence between the random variables. 

This concept of further information is in fact natural when it comes to contexts where there is an acceptable structural theory guiding the researcher to believe that the mean independence assumption is appropriate in a broad sense. 
For an elementary example,
if there is in the  researcher's field a generally accepted  theory saying that the random variables $Y, X_{1}, \dots, X_{k}$ concerning her may jointly admit some  $(k+1)$-normal distribution, 
then she may rest assured that what she will learn via stochastic linear regression is precisely the conditional expectation (modulo a translation) of $Y$ given $X_{1}, \dots, X_{k}$. 
Another elementary example is a prototypical context of time series analysis. In a chemical experiment under a suitably controlled environment, let the scalar outcome be recorded according to the natural order of time. Given the relative stability of the experimental environment, the outcome may be described by a discrete-time stationary process in some suitable sense.
If $Y$  represents the outcome of interest, 
if $X$ represents the ``lagged version'' of $Y$,
and if how $Y$ depends on $X$ is sought after,
then, since the environment is relatively stable, it would be reasonable to assume the mean independence of error term with respect to $X$ or even the independence of error term and $X$ with error term having mean zero. 

As an example regarding the appropriateness of the mean independence assumption with respect to a more special  structural theory and under a relatively uncontrolled, observational  environment, 
we wish to refer the reader to a field such as mathematical finance. 
There is in mathematical finance  the so-called efficient market hypothesis, whose empirical validity is generally acknowledged in some  special cases,  
such that, for example,  the researcher may consider a process of stock price as a martingale.
This special background structural theory,
when appropriately interpreted and applied,
then assures that the researcher's study via stochastic linear regression, for $Y$ being, say, the changes in stock prices and for $X$ being, say, a random variable measurable with respect to the ``history'' or all the ``past information'', may consider reasonable the mean independence assumption.   

In contrast with the case where a tenable structural theory is absent so that a conditional mean interpretation for stochastic linear regression may very well be  inappropriate,
we see that linear orthogonal projection is potentially indeed affine conditional expectation in the presence of a tenable strucutral theory and hence, whenever the data suggest the suitability of the estimated orthogonal projection, one may be confident in addressing the estimated coefficient as the effect on the (conditional) average behavior of the dependent variate.
Moreover, we have shown that the association of linear orthogonal projection with stochastic linear regression is equivalent to that of affine conditional expectation with stochastic linear regression (which, as seen, may ``easily'' be false by Corollary 1) if and only if the mean independence of error term holds in a reasonable way (which may be the case under a suitable background structural hypothesis). 
Thus, from a pure mathematical viewpoint without any practical consideration such as taking into account a background structural theory, 
the concept of stochastic linear regression itself need not involve conditional expectation at the outset. In particular, one cannot expect a descriptive data analysis, which is by definition a purely ``data-driven'' study without any reference to any structural theory, 
via stochastic linear regression to admit an interpretation of conditional mean.       

Structural theory plays a role that goes beyond the aforementioned matter.
We might need to stress this, that, probably due to the fact that many observational studies, in contrast with experimental studies, are implemented with a  background structural theory in mind, there is a tendency to mix the structural considerations with the purely mathematical considerations when it comes stochastic linear regression (e.g. considering the concept of ``parameter'' or of ``error'').
By a structural consideration we refer to a situation where a background structural theory, maintained or to be tested,  suggests a specific way of dependence between the variates of interest to the researcher.
As an immediate example, analyzing financial data is usually and conceivably \textit{a priori} tied to the background theory regarding the financial variates under consideration, 
and the structural theory may impose a  mathematical dependence structure for the variates.

\subsection{Preliminary Developments}
In view of the previous analysis, we see that, to ensure the two minimum requirements for the researcher's study via stochastic linear regression to be meaningful, it suffices to impose the orthogonality of error term along with ``one and a half'' regularity conditions on the variates $Y, X_{1}, \dots, X_{k}$,
i.e.  along with the conditions that the random variables $Y, X_{1}, \dots, X_{k}$ are $L^{2}$
and, e.g. that the set $\{ X_{1}, \dots, X_{k} \}$ is orthogonal. And these  sufficient conditions are reasonably mild both from the mathematical viewpoint and a practical viewpoint. Indeed, besides the simplicity of the conditions,
there is a simple but deeper mathematical reason, which is seemingly seldom stressed or even noticed, to justify the conditions and reveal a further connection among them:

\begin{fact}
If $H$ is an inner product space over $\R$, 
if $g \in H$, and if $\{ f_{1}, \dots, f_{k} \} \subset H\setminus \{ 0 \}$ is orthogonal,
then there is exactly one $(b, h) \in \R^{k} \times H$ such that $g = \sum_{j=1}^{k}f_{j}b_{j} + h$ and $h$ is orthogonal to each $f_{j}$. 
\qed
\end{fact}
Fact 1 is easily found in, e.g. the introductory textbooks of abstract algebra, and a proof of Fact 1 is apparent. If $\langle \cdot, \cdot \rangle$ denotes the inner product of $H$, 
the unique choice of $b$ is simply $b \ceq (\langle f_{j}, f_{j} \rangle^{-1}\langle f_{j}, g \rangle)_{j=1}^{k}$, and that of $h$ is simply $g - f^{\t}b$; here $f \ceq ((f_{j})_{j=1}^{k})^{\t}$.

If $H$ is the $L^{2}$ space of random variables on a given probability space,
then $H$ is an inner product space by considering the inner product $(f,g) \mapsto \int fg = \E fg$ defined on $H \times H$. 
It then follows immediately from Fact 1 that 
\begin{pro}[``abundance'' of orthogonal projection]
If $Y, X_{1}, \dots, X_{k}$ are $L^{2}$ random variables,
if no $X_{j}$ is concentrated on $\{ 0 \}$,
and if $\{ X_{1}, \dots, X_{k}\}$ is orthogonal,
i.e. if $\E X_{j}X_{\lbr{j}} = 0$ for all $1 \leq j \neq \lbr{j} \leq k$,
then there is exactly one $\beta \in \R^{k}$ and there is exactly one $L^{2}$ random variable $\eps$ such that $Y = \sum_{j=1}^{k}X_{j}\beta_{j} + \eps$ and $\E X_{j}\eps = 0$ for all $1 \leq j \leq k$, and $(\E XX^{\t})^{-1}\E XY$ with $X \ceq (X_{1}, \dots, X_{k})^{\t}$ is the unique choice of $\beta$.
\end{pro}
\begin{proof}
The first conclusion is a special case of that of Fact 1; the assumption that no $X_{j}$ is essentially zero  prevents any $X_{j}$ from being equal to $0$ with zero probability. 
For the second conclusion, we notice  that $\E XX^{\t}$ is by the orthogonality assumption a diagonal matrix with each diagonal entry nonzero and hence invertible.
\end{proof}

 Since the assumptions of Proposition 1 may be considered mild for practical purposes, in practice we can by Proposition 1  ``always'' talk about the linear orthogonal projection of a random variable given a random vector. Moreover, we remark that the orthogonality of error term actually  follows from the aforementioned regularity conditions on the involved random variables, and that the unique choice of the coefficient $\beta$ of linear orthogonal projection is precisely the familiar  ``population counterpart'' of least squares estimator. 
 Proposition 1 implies 
\begin{cor}[orthogonal projection coefficient as optimizer]
Under the assumptions of Proposition 1 with the same notation,
there is exactly one $\beta \in \R^{k}$ such that $\beta \in \argmin_{\lbr{\beta} \in \R^{k}}\E(Y - X^{\t}\lbr{\beta})^{2}$, namely, the minimization problem admits a unique solution, and $\beta = (\E XX^{\t})^{-1}\E XY$. 
\end{cor}
\begin{proof}
Under the given assumptions, writing $(\E XX^{\t})^{-1}$ $\E XY$ is legitimate. 
Since a point $\lbr{\beta}$ of $\R^{k}$ is a solution to the minimization problem only if $\lbr{\beta} = (\E  XX^{\t})^{-1}$
$\E XY$, 
there is at most one such $\lbr{\beta}$.
But the orthogonal projection coefficient $(\E XX^{\t})^{-1}$ $\E XY$ is also a solution to the minimization problem, 
there is at least one such $\lbr{\beta}$.
\end{proof}

Corollary 2 says that the coefficient of linear  orthogonal projection minimizes the approximation error in the $L^{2}$ or mean-square sense. 
Owing to  results sharing the same conclusion of Corollary 2, some authors would define orthogonal projection as best linear predictor in the mean-square sense.
It can easily be shown that a conditional expectation happens to be the best mean-square predictor, and so Corollary 2 may explain the (unjustified, as shown previously) conditional expectation interpretation of stochastic linear regression.

Further, as far as the purpose of ensuring that it is meaningful to talk about learning about the orthogonal projection coefficient, the  orthogonality condition on $\{ X_{1}, \dots, X_{k} \}$ does not cost one too much  generality.  
If $\{ X_{1}, \dots, X_{k} \}$ is as in Proposition 1, and if in addition the elements are centered, 
then the variance of every nontrivial linear combination of $\{ X_{1}, \dots, X_{k} \}$ is $> 0$. For the following local purpose(s), a collection of $L^{2}$ random variables $X_{1}, \dots, X_{k}$ is said to be \textit{essentially linearly independent} if and only if the variance of  $a_{1}X_{1} + \cdots + a_{k}X_{k}$ is $>0$ for every nonzero $(a_{1}, \dots, a_{k}) \in \R^{k}$.    
Consider the following 

\begin{pro}[orthogonalization]
If $X$ is a random element of $\R^{k}$ with $L^{2}$ components that are essentially linearly independent, then there is exactly one element $A$ of the classical Lie group $SL_{k}(\R)$ such that the components of the random element  $AX$ of $\R^{k}$ form an orthogonal set. 
\end{pro}
\begin{proof}
The proof idea is just an application of conventional wisdom (Gram-Schmidt).

For $k=1$, taking $A$ to be the matrix $[ 1 ]$ having $1$ as the single entry suffices as $\{ X \}$ is trivially or vacuously orthogonal. 

We prove for $k = 2$; the underlying machinery will then be clear for all $k \geq 3$.
Let $\lbr{X}_{1} \ceq X_{1}$. If $\lbr{X}_{2} = aX_{1} + X_{2}$, then $\E \lbr{X}_{2}\lbr{X}_{1} = 0$ implies that $a = -\E X_{1}X_{2}/\E X_{1}^{2}$. But for this particular choice of $a$, it evidently holds that $\lbr{X}_{2} \ceq aX_{1} + X_{2}$ is orthogonal to $\lbr{X}_{1} = X_{1}$.
Therefore,
the matrix 
\[
\begin{bmatrix}
1 & 0\\
a & 1
\end{bmatrix}
\]
is the unique choice of $A$. Since the unique choice of $A$ has  determinant $1$,
it lies in $SL_{k}(\R)$.

For $k \geq 3$, 
we solve the corresponding $k-1$ equations in $k-1$ unknowns.
\end{proof}

Now, if $ X_{1},$ $\dots, X_{k}$ are $L^{2}$ random variables not concentrated on $\{ 0 \}$,
and if $\{ X_{1},$ $\dots, X_{k} \}$ is not orthogonal,
then Proposition 1 is not directly applicable.
But, regarding the coefficient learning purpose,
we can by Proposition 2 apply Proposition 1 to the orthogonalized version of $X$ provided that the components of $X$ are essentially linearly independent; then the  orthogonal projection coefficient of $Y$ given $X$ is simply the orthogonal projection coefficient of $Y$ given the orthogonalized $X$  left-multiplied by the  transpose of the orthogonalization matrix. 
That the orthogonality of error term to $X$ is taken care of is due to the fact that the orthogonalization matrix has constant entries. 
Thus we arrive at

\begin{thm}[``strengthened''  Proposition 1]
If $Y, X_{1}, \dots, X_{k}$ are $L^{2}$ random variables, 
and if $\{ X_{1}, \dots , X_{k} \}$ is essentially    linearly independent, 
then there is exactly one $\beta \in \R^{k}$ and there is exactly one $L^{2}$ random variable $\eps$ such that $Y = \sum_{j=1}^{k}X_{j}\beta_{j} + \eps$ and $\E X_{j} \eps = 0$ for all $1 \leq j \leq k$.
\end{thm}
\begin{proof}
We have hinted the essential considerations. If $\{ X_{1}, \dots, X_{k} \}$ is orthogonal, then Proposition 1 implies the desired conclusions. If not, we apply Proposition 2 to orthogonalize it by a unique matrix $A \in SL_{k}(\R)$. Write $X \ceq (X_{1}, \dots, X_{k})^{\t}$; then  the components of $AX$ are all $L^{2}$; and so by Proposition 1 there is exactly one $\alpha \in \R^{k}$ and there is exactly one $L^{2}$ random variable $\delta$  such that $Y = (AX)^{\t}\alpha + \delta$ and $\E ((AX)\delta) = ((0)_{j=1}^{k})^{\t}$.
Since $((0)_{j=1}^{k})^{\t} = \E ((AX)\delta) = A(\E X\delta)$,
and since $A$ is invertible,
we have $\E X\delta = ((0)_{j=1}^{k})^{\t}$.
Taking $\beta \ceq A^{\t}\alpha$ completes the proof.
\end{proof}

\begin{rem}
We have shown that it is quite ``easy'' to ensure that the researcher's study via stochastic linear regression is meaningful. Since the non-degenerateness of each $X_{j}$ is almost automatically satisfied with respect to practical purposes, the most ``stringent'' assumption turns out to be the $L^{2}$-ness (and, perhaps, the essential  linear independence) of the involved random variables! And for the involved random variables to be $L^{2}$, if not automatically true in practice, is a very mild  condition.

Nevertheless, Theorem 2 is perhaps only of theoretical interest; without orthogonality,
the usual form of orthogonal projection coefficient is not guaranteed. \qed
\end{rem}

Now we can in passing clarify this common condition on error term in a context of stochastic linear regression, that the mean of error term is $=0$. From a mathematical viewpoint, that error term has zero mean is immaterial. If $Y, X_{1}, \dots, X_{k}$ are $L^{2}$ random variables with each $X_{j}$ not concentrated on $\{ 0 \}$,
and if there is some $1 \leq j \leq k$ such that $X_{j} = 1$, i.e. if the constant regressor $1$ is present,  then the corresponding unique error $\eps$ is by Theorem 2 orthogonal to $X_{j}$; it follows that $\E X_{j}\eps = \E \eps = 0$. 

Another closely related unclarity  regarding the uncorrelatedness between  $\eps$ and each $X_{j}$   may now be settled as well.
For convenience, we state the following elementary 

\begin{fact}[equivalence of orthogonality and uncorrelatedness under mean zero]
If $X, \eps$ are $L^{2}$ random variables, and if $\E \eps = 0$, then $\E X\eps = 0$ if and only if $\K (X, \eps) = 0$. \qed
\end{fact}

Thus orthogonality between random variables is equivalent to their uncorrelatedness when one of the random variables has mean zero. 
Further, we have

\begin{pro}[uncorrelatedness and orthogonal projection]
Under the assumptions of Proposition 1 with the same notation, if in addition there is some $1 \leq j \leq k$ such that $X_{j} = 1$, 
then $\K (X_{j}, \eps) = 0$ for all $1 \leq j \leq k$.
\end{pro}
\begin{proof}
As argued in a previous paragraph, Proposition 1 ensures that $\E \eps = 0$ in the presence of the constant regressor. The desired conclusion then follows from Fact 2. 
\end{proof}

We may proceed to the desired clarification. The issue is that sometimes in a context of stochastic linear regression the uncorrelatedness of $\eps$ and each $X_{j}$ is instead stressed without a specific reference to the orthogonality of $\eps$ to each $X_{j}$. 
As we have seen, 
the orthogonality of $\eps$ to each $X_{j}$ is of fundamental concern. And a  random variable being centered, i.e. a random variable with mean zero,
has nothing to do with the orthogonality of the random variable to a given random variable (e.g. a standard normal random variable is centered but not orthogonal to itself). Indeed, the concept of orthogonality between two random variables is ``very independent'' of that of uncorrelatedness of the random variables in the sense that they do not imply each other in abundant cases:

\begin{thm}[non-equivalence between orthogonality and uncorrelatedness]
There are continuum-many $L^{2}$ random variables $X, \eps$ such that $\E X \eps = 0$ and $\K(X, \eps) \neq 0$; and there are continuum-many $L^{2}$ random variables $X, \eps$ such that $\E X \eps \neq 0$ and $\K (X, \eps) = 0$.
\end{thm}

\begin{proof}
For both assertions,
let $t > 0$. 

If $\xi \sim N(0, t)$, let $X \ceq \xi + \sqrt{t}$ and $\eps \ceq \xi - \sqrt{t}$.
Then $\E X \eps = (\E \xi^{2}) - t = 0$. But  $\E X \eps - (\E X )(\E \eps) = 0 + t = t > 0$; so $\K (X, \eps) \neq 0$. 
Since the set $\R_{++}$ is in bijection with $\R$, the first assertion follows.

For the second assertion, let $\xi$ be a Rademacher random variable, i.e. let $\xi \sim \frac{1}{2}(\D^{-1} + \D^{1})$. Then $\E \xi = 0$ and $\E \xi^{2} = 1$.
These equalities follow directly from the definition of Lebesgue integration.
If $X \ceq t\xi^{2}$, and if $\eps \ceq (t\xi)^{-1} + t$,
then 
\[
\E X \eps &= \E (\xi + t^{2}\xi^{2})\\
&= 0 + t^{2}\cdot 1\\
&= 
t^{2}\\
&> 0.
\]
On the other hand,
we have $\E t\xi^{2} = t$ and $\E ((t\xi)^{-1} + t) = 0 + t = t$; so 
$(\E X)(\E \eps) = \E X\eps$,
and hence $\K (X, \eps) = 0$;
this completes the proof.
\end{proof}

Given Proposition 1 and Theorem 3, we see the mathematical danger of mixing orthogonality with uncorrelatedness. 
This mixture seems especially customary  when stochastic linear regression is spoken of in the presence of  an ``obvious''  background structural theory with the ``understanding'' that ``error'' means both the unobservable random disturbance to the system and the error associated with the orthogonal projection under consideration.

 The indicated issue entails a typical source of confusion when it comes to stochastic linear regression --- mixing purely mathematical concepts and considerations with colloquial ideas and purpose-specific  concerns.
 We hope that our treatment of stochastic linear regression as a whole would also help to clarify the issues of such a type.


\subsection{Stochastic Linear Regression}

Given the importance and convenience of Proposition 1, 
let us agree on
\begin{defi}(fundamental random vector and canonical error)
Let $k \in \N$;
let $Y, X_{1},$ $\dots, X_{k}$ be random variables defined on the same probability space. 
Then the random vector $(Y, X_{1}, \dots, X_{k})$ is called a \textit{fundamental random vector} if and only if i) each component of it is $L^{2}$, ii) each $X_{j}$ is not concentrated on $\{ 0 \}$, and iii) the set $\{ X_{1}, \dots, X_{k} \}$ is orthogonal. 
The difference obtained by subtracting  $Y$ from  its (unique) orthogonal projection given $X_{1}, \dots, X_{k}$ 
is called the \textit{canonical error} of the fundamental random vector.

By the \textit{orthogonal projection coefficient of a fundamental random vector} $(Y, X_{1},$ $\dots, X_{k})$ we mean the orthogonal projection coefficient of $Y$, the first component, given $X_{1}, \dots, X_{k}$, the last $k$ components.
\qed
\end{defi}

\begin{note}
The existence of a fundamental random vector is never a problem; for each $k \in \N$, one can   always consider at least the  multi-normal (Gaussian) distributions on $\B_{\R^{1+k}}$ with suitable dependence structure. \qed 
\end{note}

Thus, for every fundamental random vector, it holds by Proposition 1 and Corollary 2 that the orthogonal projection coefficient of the fundamental random vector is precisely the optimizer minimizing the $L^{2}$-norm of the canonical  error of the fundamental random vector.   

The introduced terminologies in Definition 1 will not be merely nominal; they help to fix the concepts. Their usefulness will be seen.

On the basis of all the previous remarks, in particular of Proposition 1 and Corollary 2, let us also agree on

\begin{defi}[stochastic linear regression]
Let $k \in \N$; let
$\MD_{reg}^{1,k}$ be the collection of all $\P \in \Pi(\B_{\R^{1+k}})$ such that  
i) $\P_{\pi_{j}} \neq \P_{\pi_{j}}\rceil_{\{ 0 \}}$ 
for all $2 \leq j \leq k+1$, ii)
$\int_{\R} x^{2} \df \P_{\pi_{j}}(x) < +\infty$
for all 
$1 \leq j \leq k+1,$ and 
iii) $\int_{\R^{2}} x\lbr{x} \df \P_{\pi_{\{j, \lbr{j} \}}}(x, \lbr{x}) = 0$ for all $2 \leq j \neq \lbr{j} \leq k+1$.
Then $\MD_{reg}^{1,k}$ is called a \textit{stochastic linear regression model}. 

Let $\P \in \Pi(\B_{\R^{1+k}})$. Then $\P$ is called a \textit{stochastic linear regression} if and only if $\P \in \MD^{1,k}_{reg}$.
\qed
\end{defi}

The requirements in Definition 2 are precisely and simply translated from the assumptions of Proposition 1:

\begin{thm}[characterizing stochastic linear regression model via fundamental random vector]
If $k \in \N$,
then 
\[
\MD_{reg}^{1,k} = \{ \P \in \Pi(\B_{\R^{1+k}}) \mid Z \sim \P \lrtx{for some fundamental random vector} Z \}.
\]
\end{thm}
\begin{proof}
We first prove the inclusion relation  $\supset$.
Let $\P$ be an element of the right-side collection. Then there are some probability space $(\Omega, \F, P)$ and some  fundamental random vector $Z$ on $\Omega$ such that $P_{Z} = \P$.
Since each component of $Z$ lies in $L^{2}(P)$ by Definition 1,
we have
\[
|Z_{j}|_{L^{2}(P)}^{2}
=
\E (Z_{j})^{2} = \int_{\R} x^{2} \df P_{Z_{j}}(x) = \int_{\R} x^{2} \df \P_{\pi_{j}}(x)
\]
for all $1 \leq j \leq k+1$. Here $| \cdot |_{L^{2}(P)}$ denotes the in-context $L^{2}$-norm. 

Since $Z_{j}$ is not essentially zero for all $2 \leq j \leq k+1$ by Definition 1,
it follows that 
\[
\P_{\pi_{j}} \neq \P_{\pi_{j}}\rceil_{\{ 0 \}}
\]
for all $2 \leq j \leq k+1$.

Moreover, from the orthogonality requirement of Definition 1,
we have
\[
0 
&= \E Z_{j}Z_{\lbr{j}}\\
&=
\int_{\R^{2}} 
x\lbr{x} \df P_{(Z_{j}, Z_{\lbr{j}})}(x, \lbr{x})\\
&=
\int_{\R^{2}} x\lbr{x} \df \P_{\pi_{\{ j, \lbr{j} \} }}(x, \lbr{x})
\]
for all $2 \leq j \neq \lbr{j} \leq k+1$. 
The last equality  follows jointly  from the facts that \[
\sigma(\{ B_{1} \times B_{2} \mid B_{1}, B_{2} \in \B_{\R}
\}) = \B_{\R^{2}}
\]
and that the collection $\{ B_{1} \times B_{2} \mid B_{1}, B_{2} \in \B_{\R} \}$ is stable with respect to finite intersections. The inclusion $\supset$ follows.

For the other inclusion $\subset$, let $\P \in \MD_{reg}^{1,k}$. If $\Omega \ceq \R^{1+k}$, if $\F \ceq \B_{\R^{1+k}}$, and if $P \ceq \P$, 
then,
upon taking $Z$ to be the $(k+1)$-tuple  $(\pi_{1}, \dots, \pi_{k+1})^{\t}$ of natural projections $\pi_{j}$ defined on $\Omega$,
the desired inclusion relation  $\subset$ follows. 
\end{proof}

From a mathematical viewpoint, we have obtained the desired intrinsic treatment of stochastic linear regression in the old sense on the basis of Proposition 1, Corollary 2, and Theorem 4.
We have established the concept of a stochastic linear regression in a way involving and only involving pure mathematical considerations. Since the stochastic linear regression model $\MD^{1,k}_{reg}$ is for every $k \in \N$ defined as a subcollection of the class $\Pi(\B_{\R^{1+k}})$ of all probability measures on $\B_{\R^{1+k}}$,
and since a stochastic linear regression, being defined as an element of $\MD_{reg}^{1,k}$ for some $k \in \N$, 
is simply a probability measure with suitable properties,
we have embedded the concept of stochastic linear regression (in the old sense) in mathematics in the sense discussed in the introduction of the present paper.

Besides the elegance and  conceptual utilities of defining a stochastic linear regression as a probability measure, we wish to liken the present  situation to an existing one in the related literature, although the treatment in the  existing situation is much less exotic, and nearly requires no further justifications or  elaborations: In the  literature of machine learning, some authors take a stochastic process to be a probability measure. 
We refer the reader to, e.g. Ryabko \cite{r} or Khaleghi and Ryabko \cite{kr}.       

With respect to practical purposes, as we have argued previously how a fundamental random vector is the basic ingredient or ``regression material'' for stochastic linear regression in the old sense, 
Theorem 4 preserves this important aspect of stochastic linear regression (in the old sense), 
and hence serves as a further justification for our definition of a stochastic linear regression model. 

Nevertheless, in practice there are certainly a large number of  situations where one does not and cannot reasonably consider an arbitrary fundamental random vector as the basic ``regression material''. For studies with a reference to a structural theory, the first component of the fundamental random vector under consideration usually depends on the last $k$ components of the fundamental random vector in a pre-specified way that is suggested by the reference  structural theory; in such a case, the first component is then defined in terms of the last $k$ components. In fact, this pre-specified dependence is the basis of  any simulation studies involving stochastic linear regression in the old sense; given $k$ orthogonal random variables $X_{1}, \dots, X_{k}$ being $L^{2}$ and not concentrated on $\{ 0 \}$,
some 
$k$ constants $\neq 0$,
and an additional ``well-behaved'' $L^{2}$  random variable $\eta$ independent of each $X_{j}$ in some suitable way,
one defines a new random variable $Y$ as the linear combiniation of the random variables $X_{1}, \dots, X_{k}$ and the constants plus the random variable  $\eta$. 
[If $\eta$ is $L^{2}$, then $Y$ is also $L^{2}$ by Minkowski's inequality; so $(Y, X_{1}, \dots, X_{k})$ is a fundamental random vector.]

Excluding the simulation studies,
a context where a dependence is pre-specified within the fundamental random vector under consideration is another significant source of confusion when it comes to stochastic linear regression in the old sense, especially when the definition of $Y$ in terms of $X_{1}, \dots, X_{k}$ and $\eta$ has the same form as the inherent orthogonal projection of $Y$ given $X_{1}, \dots, X_{k}$ plus the inherent canonical error $\eps$.   But the  concept of  ``imposed error'' or ``structural error''  $\eta$ is independent of that of canonical error $\eps$; 
structural error need not be orthogonal to each $X_{j}$.

Our notion of a stochastic linear regression model takes care of the case where a fundamental random vector considered in a simulation study is obtained by the linear-additive  specification; we have another characterization of a stochastic linear regression model:

\begin{thm}[parametrized representation of stochastic linear regression model]
Let $k \in \N$. For every $\beta \in \R^{k}$, let $\mathscr{S}_{\beta}$ be the collection of all $\P \in \Pi(\B_{\R^{1+k}})$ such that i) $(\pi_{1}, \dots, \pi_{k+1})$ is a fundamental random vector on the probability space $(\R^{1+k}, \B_{\R^{1+k}}, \P)$ and ii) there is exactly one
$\eta \in L^{2}(\P)$ such that $\pi_{1} = \sum_{j=2}^{k+1}\pi_{j}\beta_{j} + \eta$ and $\E \pi_{j}\eta = 0$ for all $2 \leq j \leq k+1$.
Then $\beta, \lbr{\beta} \in \R^{k}$ and $\beta \neq \lbr{\beta}$ imply $\mathscr{S}_{\beta} \cap \mathscr{S}_{\lbr{\beta}} = \varnothing$, and 
\[
\MD_{reg}^{1,k} = \bigcup_{\beta \in \R^{k}}\mathscr{S}_{\beta}.
\]
\end{thm}
\begin{proof}
To see the disjointness, let $\beta, \lbr{\beta} \in \R^{k}$ be distinct, and suppose there is some $\P \in \Pi(\B_{\R^{1+k}})$ such that $\P \in \mathscr{S}_{\beta} \cap \mathscr{S}_{\lbr{\beta}}$. Then $(\pi_{1}, \dots, \pi_{k+1})$ is a fundamental random vector with respect to $\P$, and so by Proposition 1 we have $\beta = \lbr{\beta}$, a contradiction. 

The inclusion $\supset$ follows trivially from the definition of $\mathscr{S}_{\beta}$, and the inclusion $\subset$ follows  from Proposition 1 and the proof of Theorem 4.
\end{proof}

Theorem 5 shows an  additional nice feature of our definition of a stochastic linear regression model. Moreover, it allows us to connect an indexed family of stochastic linear regressions with the  notion of  parameter identifiability:
\begin{thm}[parametrization injectiveness of  stochastic linear regression]
Let $k \in \N$; let  $\mathscr{S}_{\beta}$ be the same as in Theorem 5 for all $\beta \in \R^{k}$.
If $\Theta \subset \R^{k}$,
and if $(\P^{\beta})_{\beta \in \Theta} \in \bigtimes_{\beta \in \Theta}\mathscr{S}_{\beta}$,
then the map $\beta \mapsto \P^{\beta}, \Theta \to \{ \P^{\beta} \mid \beta \in \Theta \}$ is an injection.
\end{thm}
\begin{proof}
First of all, every $\mathscr{S}_{\beta}$ is nonempty; acknowledging the axiom of choice implies that our assumption is not vacuous. From the first assertion of Theorem 5, it follows that $\P^{\beta} \neq \P^{\lbr{\beta}}$ for all $\beta, \lbr{\beta} \in \Theta$ such that $\beta \neq \lbr{\beta}$.  
\end{proof}

Theorem 6 ensures that a collection of fundamental random vectors, obtained by specifying the  linear-additive dependence as the constant vector runs through a given subset of $\R^{k}$,  is indeed a collection of distinct stochastic linear regressions:
\begin{cor}[simulation and stochastic linear regression]
Let $\eta, X_{1}, \dots, X_{k}$ be $L^{2}$ random variables defined on the same probability space; 
let $X_{j} \not\sim \D^{0}$ for all $1 \leq j \leq k$; let $\{\eta,  X_{1}, \dots, X_{k} \}$ be orthogonal;
let $\Theta \subset \R^{k}$.
For every $\beta \in \Theta$,
let $Y \ceq \sum_{j=1}^{k}X_{j}\beta_{j} + \eta$, so that $(Y, X_{1}, \dots, X_{k})$ is a fundamental random vector. If $\P^{\beta}$ is the distribution of $(Y, X_{1}, \dots, X_{k})$ for all $\beta \in \Theta$,
i.e. if $\P^{\beta}$ is the stochastic linear regression corresponding to $(Y,X_{1}, \dots, X_{k})$ for all $\beta \in \Theta$,
then the map $\beta \mapsto \R^{\beta},\Theta \to \{ \P^{\beta} \mid \beta \in \Theta \}$ is an injection.\qed
\end{cor}

Thus, for any given fundamental random vector subject to the specification given in Corollary 3, 
different given values of $\beta$ determine different stochastic linear regressions. This conclusion is certainly desirable for the practical purposes. 

We have completed our intended  treatment.

\end{document}